\documentclass[10pt,a4paper]{article}
\usepackage{latexsym,amsthm,amssymb,amscd,amsmath}
\newcounter{alphthm}
\theoremstyle{plain}
\newtheorem{theorem}{Theorem}[section]

\theoremstyle{definition}
\newtheorem{defn}[theorem]{Definition}

\newcommand{\be}{\begin{equation}}
\newcommand{\ee}{\end{equation}}
\newcommand{\ben}{\begin{enumerate}}
\newcommand{\een}{\end{enumerate}}

\begin{document}
\title{On $\phi$-recurrent generalized Sasakian-space-forms}
\author{E. Peyghan and A. Tayebi}
\maketitle

\maketitle
\begin{abstract}
We show that there is no $\phi$-recurrent generalized Sasakian-space-forms, when $f_1-f_3$ is a non-zero constant.
\end{abstract}

\section{Introduction}
In differential geometry, the curvature of a Riemannian manifold $(M, g)$ plays a
fundamental role, and, as is well known, the sectional curvatures of a manifold
determine the curvature tensor $R$ completely. A Riemannian manifold with constant sectional
curvature $c$ is called a real-space-form, and its curvature tensor satisfies the
equation
\[
R(X, Y)Z=c\{g(Y, Z)X-g(X, Z)Y\}.
\]
Models for these spaces are the Euclidean spaces $(c=0)$, the spheres $(c>0)$ and the hyperbolic spaces $(c<0)$.

A Sasakian manifold $(M, \phi, \xi, \eta, g)$ is said to be a Sasakian-space-form, if all the $\phi$-sectional curvatures $K(X\wedge\phi X)$ are equal to a constant $c$, where $K(X\wedge\phi X)$ denotes the sectional curvature of the section spanned by the unit vector field $X$ orthogonal to $\xi$ and $\phi X$. In such a case, the Riemannian curvature tensor of $M$ is given by
\begin{align}\label{alah}
R(X, Y)Z&=\frac{c+3}{4}\{g(Y, Z)X-g(X, Z)Y\}+\frac{c-1}{4}\{g(X, \phi Z)\phi Y\nonumber\\
&\ \ \ -g(Y, \phi Z)\phi X+2g(X, \phi Y)\phi Z\}+\frac{c-1}{4}\{\eta(X)\eta(Z)Y\nonumber\\
&\ \ \ -\eta(Y)\eta(Z)X+g(X, Z)\eta(Y)\xi-g(Y, Z)\eta(X)\xi\}.
\end{align}
As a natural generalization of these manifolds, P. Alegre, D. E. Blair and A. Carriazo \cite{ABC} introduced the notion of generalized Sasakian-space-form. It is defined as almost contact metric manifold with Riemannian curvature tensor satisfying an equation similar to (\ref{alah}), in which the constant quantities $\frac{c+3}{4}$ and $\frac{c-1}{4}$ are replaced by differentiable functions.

Local symmetry is a very strong condition for the class of Sasakian manifolds. Indeed, such spaces must have constant curvature equal to 1 \cite{O}.
Thus Takahashi introduced the notion of a (locally) $\phi$-symmetric space
in the context of Sasakian geometry \cite{T}.
Generalizing the notion of $\phi$-symmetry, De-Shaikh-Biswas  introduced the notion of $\phi$-recurrent Sasakian manifold \cite{DSB}. Then Sarkar-Sen studied the notion of $\phi$-recurrent for Generalized Sasakian-Space-Forms \cite{SS}. They deduce some results when $f_1-f_3$ is a nonzero constant and the dimension of the manifold is bigger than $3$ . For example the following theorem is one of them.
\begin{theorem}\cite{SS}
A ö$\phi$-recurrent generalized-Sasakian-space-form $(M^{2n+1}, g)$ is an Einstein manifold
provided $f_1-f_3$ is a non-zero constant.
\end{theorem}
In this paper we show that there is no $\phi$-recurrent generalized Sasakian-space-forms that the difference of  $f_1$ and $f_3$ is a nonzero constant at all, when the dimension of the manifold is bigger than $3$ .
\section{Contact Metric Manifolds}
We start by collecting some fundamental material about contact metric geometry. We refer to \cite{B}, \cite{BKP} for further details.

A differentiable $(2n + 1)$-dimensional manifold $M^{2n+l}$ is called a {\it{contact manifold}} if it carries a global differential 1-form $\eta$ such that $\eta\wedge(d\eta)^n\neq 0$ everywhere on $M^{2n+1}$. This form $\eta$ is usually called the {\it{contact form}} of $M^{2n+1}$. It is well known that a contact manifold admits an {\it{almost contact metric structure}} $(\phi, \xi, \eta, g)$, i.e., a global vector field $\xi$, which will be called the {\it{characteristic vector field}}, a (1, 1) tensor field $\phi$ and a Riemannian metric $g$ such that
\begin{align}
&(i)\ \eta(\xi)=1,\ \ \ (ii)\ \phi^2=-Id+\eta\otimes\xi,\label{con}\\
&g(\phi X, \phi Y)=g(X, Y)-\eta(X)\eta(Y),\label{con2}
\end{align}
for any vector fields on $M^{2n+1}$. Moreover, $(\phi, \xi, \eta, g)$ can be chosen such that $d\eta(X, Y)=g(X, \phi Y)$ and we then call the structure a {\it contact metric structure} and the manifold $M^{2n+1}$ carrying such a structure is said to be a {\it contact metric manifold}. As a consequence of (\ref{con}) and (\ref{con2}), we have
\[
\phi\xi=0,\ \ \ \eta\circ\phi=0,\ \ \ d\eta(\xi, X)=0.
\]
Denoting by $\pounds$, Lie differentiation, we define the operator $h$ by following
\[
hX:=\frac{1}{2}(\pounds_\xi\phi)X.
\]
The (1, 1) tensor $h$ is self-adjoint and satisfy
\begin{equation}
(i)\ h\xi=0,\ \ \ (ii)\ h\phi=-\phi h,\ \ \ (iii)\ Tr h=Tr h\phi=0.
\end{equation}
Since  the operator $h$ anti-commutes with $\phi$, if $X$ is an eigenvector of $h$ corresponding to the
eigenvalue $\lambda$, then $\phi X$ is also an eigenvector of $h$ corresponding to the eigenvalue
$-\lambda$.

If $\nabla$ is the Riemannian connection of $g$, then
\begin{equation}
\nabla_X\xi=-\phi X-\phi hX,\ \ \ \nabla_\xi\phi=0.
\end{equation}
A contact structure on $M^{2n+1}$ gives rise to an almost complex structure on
the product $M^{2n+1}\times R$. If this structure is integrable, then the contact metric
manifold is said to be {\it Sasakian}.
\begin{defn}
\emph{A contact metric manifold $M^{2n+1} (\phi, \xi, \eta, g)$ is said to be locally $\phi$-symmetric if
\begin{equation}
\phi^2((\nabla_WR)(X, Y)Z)=0,\label{a1}
\end{equation}
for all vector fields W, X, Y, Z orthogonal to $\xi$. If (\ref{a1}) holds for all vector fields $W$, $X$, $Y$, $Z$ (not necessarily orthogonal to $\xi$), then we call it $\phi$-symmetric.}
\end{defn}
The notion locally $\phi$-symmetric, was introduced for Sasakian manifolds by Takahashi \cite{T}.
\begin{defn}\label{hasan6}
\emph{A contact metric manifold $M^{2n+1} (\phi, \xi, \eta, g)$ is said to be $\phi$-recurrent if there exists a non-zero 1-form $A$ such that
\begin{equation}\label{rec}
\phi^2((\nabla_WR)(X, Y)Z)=A(W)R(X, Y)Z,
\end{equation}
for all vector fields $X, Y, Z, W$.}
\end{defn}
This notation was introduced for Sasakian manifolds by De-Shaikh-Biswas \cite{DSB}.

Given an almost contact metric manifold $M(\phi, \xi, \eta, g)$, we say that $M$ is generalized Sasakian-space-form if there exist three functions $f_1$, $f_2$, $f_3$ on $M$ such that the curvature tensor $R$ is given by
\begin{align}
R(X, Y)Z&=f_1\{g(Y,Z)X-g(X,Z)Y\}+f_2\{g(X, \phi Z)\phi Y-g(Y, \phi Z)\phi X\nonumber\\
&\ \ \ +2g(X, \phi Y )\phi Z\}+f_3\{\eta(X)\eta(Z)Y-\eta(Y)\eta(Z)X\nonumber\\
&\ \ \ +g(X,Z)\eta(Y)\xi-g(Y,Z)\eta(X)\xi\},
\end{align}
for any vector fields $X$, $Y$, $Z$ on $M$. In such a case we denote the manifold as
$M(f_1, f_2, f_3)$. In \cite{ABC} the authors cited several examples of such manifolds. If $f_1=\frac{c+3}{4}$, $f_2=\frac{c-1}{4}$ and $f_3=\frac{c-1}{4}$
 then a generalized Sasakian-space-form with Sasakian structure becomes Sasakian-space-form.

We also have for a generalized Sasakian-space-form (see \cite{ABC}, \cite{SS})
\begin{align}
&R(X, Y)\xi=(f_1-f_3)[\eta(Y)X-\eta(X)Y],\label{rep10}\\
&R(\xi, X)Y=(f_1-f_3)[g(X, Y)\xi-\eta(Y)X],\\
&\eta(R(X, Y)Z)=(f_1-f_3)(g(Y, Z)\eta(X)-g(X, Z)\eta(Y))\label{rep20}.
\end{align}
\section{Nonexistence of $\phi$-Recurrent Generalized Sasakian Space Form}
In this section, we give the main result of the paper which implies some results of \cite{SS}, for example Theorem 3.1, are incorrect.
\begin{theorem}\label{SRE1}
There is no $\phi$-recurrent generalized Sasakian space form $M^{2n+1}$, $n>1$, with $f_1-f_3\neq 0$.
\end{theorem}
\begin{proof}
Let $M^{2n+1}$ $(n>1)$, be a $\phi$-recurrent Sasakian manifold. Then using (ii) of (\ref{con}) and (\ref{rec}), we get
\[
-(\nabla_WR)(X, Y)Z+\eta((\nabla_WR)(X, Y)Z)\xi=A(W)R(X, Y)Z,
\]
or
\begin{equation}\label{Bi}
(\nabla_WR)(X, Y)Z=\eta((\nabla_WR)(X, Y)Z)\xi-A(W)R(X, Y)Z,
\end{equation}
where $X, Y, Z, W$ are arbitrary vector fields on $M$ and $A$ is a non-zero 1-form on $M$.
Using Bianchi identity
\[
(\nabla_WR)(X, Y)Z+(\nabla_XR)(Y, W)Z+(\nabla_YR)(W, X)Z=0,
\]
in (\ref{Bi}) implies that
\[
A(W)R(X, Y)Z+A(X)R(Y, W)Z+A(Y)R(W, X)Z=0.
\]
Applying $\eta$ to the above equation yields
\begin{equation}\label{eta}
A(W)\eta(R(X, Y)Z)+A(X)\eta(R(Y, W)Z)+A(Y)\eta(R(W, X)Z)=0.
\end{equation}
Since $f_1-f_2$ is non-zero constant then plugging  (\ref{rep20}) in (\ref{eta}), it follows that
\begin{align}\label{SR}
&A(W)[g(Y, Z)\eta(X)-g(X, Z)\eta(Y)]+A(X)[g(W, Z)\eta(Y)-g(Z, Y)\eta(W)]\nonumber\\
& \ \ \ \ \ \ \ \ +A(Y)[g(X, Z)\eta(W)-g(W, Z)\eta(X)]=0.
\end{align}
Now, we choose the $\phi$-basis $\big\{e_i, \phi e_i, \xi\big\}_{i=1}^n$ for $M^{2n+1}$ $(n>1)$. By setting $Y=Z=e_i$, $W=e_j$ $(j\neq i)$ and $X=\xi$ in (\ref{SR}),  we obtain
\[
A(e_j)=0.
\]
Since $j$ is arbitrary, then we deduce
\begin{equation}\label{SR1}
A(e_k)=0,\ \ \ \forall\ k=1,\ldots,n.
\end{equation}
Similarly, setting $Y=Z=e_i$, $W=\phi e_j$ and $X=\xi$ in (\ref{SR}) implies
\[
A(\phi e_j)=0.
\]
Thus we deduce
\begin{equation}\label{SR2}
A(\phi e_k)=0,\ \ \ \forall\ k=1,\ldots,n.
\end{equation}
Let $Y$ be a non-zero vector field orthogonal to $\xi$. Then (\ref{rep10}) gives us
\begin{align*}
(\nabla_\xi R)(\xi, Y)\xi&=\nabla_\xi R(\xi, Y)\xi-R(\xi, \nabla_\xi Y)\xi\\
&=(f_3-f_1)\eta(\nabla_\xi Y)\xi\label{a8}
\end{align*}
From $g(Y, \xi)=0$ we obtain $g(\nabla_\xi Y, \xi)=0$ or $\eta(\nabla_\xi Y)=0$. Therefore using the above equation we deduce
\begin{equation}\label{tela}
(\nabla_{\xi}R)(\xi, Y)\xi=0.
\end{equation}
Thus from (\ref{rec}) and (\ref{rep10}) we derive that
\[
0=(f_1-f_3)A(\xi)R(\xi, Y)\xi=(f_3-f_1)A(\xi)Y.
\]
The above equation give us $A(\xi)=0$. Thus by using (\ref{SR}) and (\ref{SR1}),  we deduce that $A=0$ on $M$, which is a contradiction.
\end{proof}

\bigskip

\noindent
Esmaeil Peyghan\\
Department of Mathematics, Faculty  of Science\\
Arak University\\
Arak 38156-8-8349,  Iran\\
Email: epeyghan@gmail.com
\bigskip

\noindent
Akbar Tayebi\\
Department of Mathematics, Faculty  of Science\\
University of Qom \\
Qom. Iran\\
Email:\ akbar.tayebi@gmail.com

\end{document}